\documentclass[10pt]{amsart}
\usepackage{amsmath}
\usepackage{amsfonts}
\usepackage{amscd}
 \usepackage{graphicx}

\newtheorem{thm*}{\bf Theorem}[section]

\newtheorem{thm}{\bf Theorem}[section]
\newtheorem{cor}{\bf Corollary}[section]
\newtheorem{lemma}{\bf Lemma}[section]
\newtheorem{prop}{\bf Proposition}[section]
\newtheorem{defn}{\bf Definition}[section]

\newtheorem{rem}{\bf Remark}[section]

\newcommand{\<}{\langle}
\renewcommand{\>}{\rangle}

\newcommand{\w}{\omega}

\newcommand{\sh}{\operatorname{sh}}
\newcommand{\ch}{\operatorname{ch}}

\renewcommand{\phi}{\varphi}

\newcommand{\A}{\operatorname{\mathfrak{a}}}

\renewcommand{\Im}{\operatorname{Im}}
\renewcommand{\Re}{\operatorname{Re}}

\newcommand{\RR}{{\mathbb R}}
\newcommand{\CC}{{\mathbb C}}

\newcommand{\NN}{{\mathbb N}}

\newcommand{\ZZ}{{\mathbb Z}}

\newcommand{\E}{{\mathcal E}}

\begin{document}

 \title{
\Large{Hyperbolic Eisenstein series  
 for geometrically finite\\
  hyperbolic surfaces of infinite volume.}
}

\author{ Th\'er\`ese \textsc{Falliero}}
\address{ Universit\'e d'Avignon et des Pays de Vaucluse,
Laboratoire de math\'ematiques d'Avignon (EA 2151),
F-84018 Avignon, France}

\email{therese.falliero@univ-avignon.fr}



\maketitle

\begin{abstract} Let $M$ be a geometrically finite hyperbolic
surface of infinite volume.  After writing down the spectral
decomposition for the Laplacian on 1-forms of $M$, we generalize
the Kudla and Millson's construction of hyperbolic Eisenstein
series 
and other related
results (see theorems \ref{dual}, \ref{infinite}, \ref{deg}, \ref{vitali}). \end{abstract}

\maketitle
\section*{ Introduction.}
\medskip 
The spectrum  of the Laplace-Beltrami operator for a compact Riemann surface is discrete, it is 
no more the case when $M$ is not compact. For example when we withdraw one point from $M$, it appears a continuous part in the spectrum whose spectral measure is described by an Eisenstein series. The study of the limiting behavior of the spectrum of the Laplace-Beltrami operator for a degenerating family of Riemann surfaces with finite area hyperbolic metric have been used to explain this apparition (see for example \cite{w2}, \cite{jit}, \cite{he}). This paper finds  
one of its motivation in the general study of approximation to Eisenstein series (see for example  the
question of L.Ji in \cite{jit}, p.308, line 15, concerning the
approximation of Eisenstein series by suitable eigenfunctions of
a degenerating family of hyperbolic Riemann surfaces). We hope to
surround it via hyperbolic Eisenstein series (for  results on
degenerating Eisenstein series  see, for example \cite{o2},
\cite{o4}, \cite{jorg}, \cite{pippich}). What we really do here
is to develop the suggestion in \cite{kmi}, to construct
hyperbolic Eisenstein series and harmonic dual forms in the
infinite volume case: in this context we verify the convergence
of hyperbolic Eisenstein series and the fact that it permits to
realize a harmonic dual form to a simple closed geodesic on a
geometrically finite hyperbolic surface of infinite volume
(theorem \ref{dual}) and in a similar way of an infinite
geodesic joining a pair of punctures (theorem \ref{infinite}).
Moreover in section 5 we generalize the definition of hyperbolic Eisenstein series to the case of $q-$forms (see section 5 formula (\ref{nor})). In  particular we obtain  a degeneration of hyperbolic Eisenstein series to
horocyclic ones (theorem \ref{vitali}(2)). Since it is in the center of our motivation let's be a little  more precise on this new result.\\
 Let $(S_l)_l$ be a degenerating family of Riemann surfaces with infinite area hyperbolic metric, $S_l$ having a funnel $F_l$ whose boundary geodesic is denoted $c_l$. We denote
$\Omega_{c_l}=\Omega_l$ the
hyperbolic Eisenstein series  of Kudla-Millson theory associated to the pinching geodesic $c_l$ (see section 2), $S_0$ the component of the limiting surface containing the cusp $\infty$ (see section 5.1) and ${\mathcal E}_\infty$ the horocyclic Eisenstein series associated to the limiting cusp (see formula (\ref{eh})).
 For $s\in\{\Re s>-1/2\}$, except possibly a finite number of points in $(-1/2,0)$, the family of meromorphic 1-forms
$\frac{1}{l^{s+1}}\Omega_l(s,.)$ converges uniformly on
compact subsets  of $S_0$ to $\frac{\Gamma\left(1+
\frac{s}{2}\right)}{\Gamma\left( \frac{1}{2}\right)\Gamma\left(
\frac{1}{2}+ \frac{s}{2}\right)}\Im{\mathcal E}_\infty(s+1,.)$.





\section{Preliminary definitions}
Let us recall the standard analytic and geometric notations which will be used. 
In this paper a surface is a connected orientable two-dimensional manifold, without boundary unless otherwise specified. We denote by $H$ the hyperbolic upper half-plane endowed with its standard metric of constant gaussian curvature $-1$.
A topologically finite (i.e. finite Euler characteristic) surface is  a surface homeomorphic to a compact surface with finitely many points excised and
a geometrically finite hyperbolic surface $M$ is a
topologically finite, complete Riemann surface of constant
curvature -1.
 We will require that $M$ is of infinite volume, 
  then there exists a finitely
generated, torsion free, discrete subgroup, $\Gamma$, of
$PSL(2,\mathbb R)$, unique up to conjugation, such that $M$ is
the quotient of $ H$ by
$\Gamma$ acting as M\"obius transformations, $\Gamma$ is a Fuchsian group of the second kind and $\Gamma$ has no elliptic elements. The group $\Gamma$ admits a finite sided polygonal fundamental domain in $H$. We recall now
the description of the fundamental
domain of $M=\Gamma \backslash H $ (see \cite{borth}).
 Let $L(\Gamma)$ be the limit
set of $\Gamma$, that is the set of limit points (in the Riemann-sphere topology) of all orbits $\Gamma z$ for $z\in H$ and $O(\Gamma)={\mathbb R }\cup \{\infty\} -
L(\Gamma)$. As $L(\Gamma)$ is closed in $\mathbb R \cup
\{\infty\}$, $O(\Gamma)$ is open and so can be written as a
countable union  $O(\Gamma)=\cup_{\alpha \in A}O_\alpha$ where
the $O_\alpha$ are disjoint open intervals in $\mathbb R \cup
\{\infty\}$. Then let $\Gamma_\alpha=\{\gamma \in \Gamma,
\gamma(O_\alpha)=O_\alpha\}$.\\ This is an elementary
hyperbolic subgroup of $\Gamma$. The fixed points of
$\Gamma_\alpha$ are exactly the end-points of $O_\alpha$. There
is a finite subset $\{\alpha(1),
\alpha(2),...,\alpha(n_f)\}\subset A$ so that, for $\alpha\in A$,
$O_\alpha$ is conjugate to precisely one $O_{\alpha(j)}$ $(1\leq
j \leq n_f)$. Let $\lambda_\alpha$ be the half-circle, lying in
$H$, joining the end-points of $O_\alpha$. Let $\Delta_\alpha$
be the region in $H$ bounded by $O_\alpha$ and $\lambda_\alpha$.
The $\Delta_\alpha$ $(\alpha \in A)$ are mutually disjoints.\\ Let
$P$ be the set of parabolic vertices of $\Gamma$, and for $p\in
P$ let $\Gamma_p$ be the parabolic subgroup of $\Gamma$ fixing
$p$.  There is a finite subset $\{p(1), p(2),...,p(n_c)\}\subset
P$ so that $\Gamma_p$ is conjugate to precisely one
$\Gamma_{p(j)}$ $(1\leq j \leq n_c)$. A circle lying in $H$ and tangent to $\partial H$ at $p$ is called a horocycle at $p$. We can construct an open disc 
$C_p$ determined by a horocycle  at $p\in P$ so that: \begin{eqnarray*} (i)&\text{if }\,
p,\, q\in P, p\not= q, \text{then }\, C_p\cap C_q=\emptyset,\\
(ii)& \gamma(C_p)=C_{\gamma(p)} \, (\gamma\in \Gamma),\\ (iii)&
C_p\cap \Delta_\alpha=\emptyset\, \, \,  (p\in P,\alpha\in A).\\
\end{eqnarray*} If we consider the set
$H-(\bigcup_{p\in P} C_p\cup \bigcup_{\alpha\in A}\Delta_\alpha)$, 
we see that it is invariant under $\Gamma$. We can find a
finite-sided fundamental domain $D$ for the action of $\Gamma$ on
this set; $D$ is relatively compact in $H$.
 \begin{prop} There is a
 fundamental domain $D$ for $\Gamma$ of
the form
$$D=K^*\cup \cup_{j=1}^{n_f}D_{\alpha(j)}\cup_{k=1}^{n_c} D_{p(k)}^*$$
where\\ 1) $K^*$ is relatively compact in $H$, \\ 2) $D_{\alpha(j)}$ is
a standard fundamental domain of $\Gamma_{\alpha(j)}$ on $\Delta_{\alpha(j)}$,\\ 3)
$D_{p(k)}^*$ is a standard fundamental domain for $\Gamma_{p(k)}$ on $C_{p(k)}$.
\end{prop}
We should note that $p\not= 0$ if and only if $\Gamma$ is of the
second kind. 

 The Nielsen region of the group $\Gamma$ is the set
$\tilde{N}=H-(\cup_{\alpha\in A} \Delta_\alpha)$, the truncated Nielsen
region of $\Gamma$ is $\tilde{K}=\tilde{N}-(\cup_{p\in P} C_p)$,
$K=\Gamma\backslash \tilde{K}$ is called the compact core of
$M$. So the surface $M=\Gamma\backslash H$ can be decomposed into a finite area surface with geodesic boundary $N$, called the Nielsen region, on which infinite area ends $F_i$ are glued: the funnels. The Nielsen region $N$ is itself decomposed into a compact surface $K$ with geodesic and horocyclic boundary on which non compact, finite area ends $C_i$ are glued: the cusps. We have
$ M=K\cup C\cup F\, ,$
where $C=C_1\cup...\cup C_{n_c}$ and $F=F_1\cup ...\cup F_{n_f}$.

A hyperbolic transformation $T\in PSL(2,\RR)$ generates a cyclic
hyperbolic group $\<T\>$. The quotient $C_l=\<T\>\backslash H$
is a hyperbolic cylinder of diameter $l=l(T)$. By conjugation,
we can identify the generator $T$ with the map $z\mapsto e^l z$,
and we define $\Gamma_l$ to be the corresponding cyclic group. A
natural fundamental domain for $\Gamma_l$ would be the region
${\mathcal F}_l=\{z\in H, 1\leq |z|\leq e^l\}$. The $y-$axis is the lift
of the only simple closed geodesic on $C_l$, whose length is $l$.
The standard  funnel of diameter $l>0$, $F_l$, is the half  hyperbolic cylinder $\Gamma_l\backslash H$, $F_l=({\mathbb R}^+)_r\times ({\mathbb R}\backslash {\mathbb Z})_x$
with the metric $ds^2=dr^2+l^2\cosh^2(r)dx^2$.\\We can always conjugate a parabolic cyclic group $\<T\>$ to the group
$\Gamma_\infty$ generated by $z\mapsto z+1$, so the parabolic
cylinder is unique up to isometry. A natural fundamental domain
for $\Gamma_\infty$ is ${\mathcal F}_\infty=\{0\leq \Re z\leq
1\}\subset H$. The standard cusp $C_\infty$ is the half parabolic cylinder $\Gamma_\infty\backslash H$, $C_\infty=([0, \infty[)_r\times ({\mathbb R}\backslash {\mathbb Z})_x$ with the metric $ds^2=dr^2+e^{-2r}dx^2$.
The funnels $F_i$ and the cusps $C_i$ are isometric to the preceding standard models.
We define the function $r$ as the distance to the
compact core $K$ and the function $\rho$ by
$$\rho(r)=\left\{\begin{array}{cc}2e^{-r}&\text{ in $ F$}\\ e^{-r}&\text{ in $C$}\end{array}\right.\, .$$
We will adopt $(\rho,t)\in(0,2]\times {\RR}/l_j{\ZZ}$ as the standard
coordinates for the funnel $F_j$, where $t$ is arc length around
the central geodesic at $\rho=2$.\\ For the cusp our standard
coordinates $(\rho, t)\in (0,1]\times {\RR}/{\ZZ}$ are based on the
model defined by the cyclic group $\Gamma_\infty$. The cusp boundary is $y=1$, so that $y=e^r$ and
$\rho=1/y$. We set  $t \equiv x \pmod \ZZ$.

\section{ Hyperbolic Eisenstein series on a geometrically finite hyperbolic surface of  infinite volume.}
\subsection{ Return to  Kudla and Millson hyperbolic Eisenstein
series' definition.}
In the following, $M$ will denote an arbitrary Riemann surface
and $L^2(M)$, the Hilbert space of square integrable 1-forms
with inner product
$$(w_1,w_2)=\frac{1}{2}\int_M w_1\wedge *\overline{w_2}\, ,$$
and corresponding norm $||.||_{L^2}$. The pointwise norm of a 1-form $w$ is defined by
$w\wedge *\overline{w}=||w||^2*1$ where $*1$ is the volume form.
\\ Let $c$ be a simple closed curve
on $M$. We may associate with $c$ a real smooth closed
differential $n_c$ with compact support such 
\begin{equation}\label{cons}
\int_c \omega =\int_M\omega\wedge n_c\, ,
\end{equation}
 for all closed differentials $\omega$.
Since every cycle $c$ on $M$ is a finite sum of cycles
corresponding to simple closed curves, we conclude that to each
such $c$, we can associate a real closed differential $n_c$
with compact support such that (\ref{cons}) holds.\\ Let $a$ and
$b$ be two cycles on the Riemann surface $M$. We define the
intersection number of $a$ and $b$ by
$$a.b=\int_M n_a\wedge n_b\, .$$

In \cite{kmi}, Kudla and Millson construct the harmonic 1-form
dual to a  simple closed geodesic on a hyperbolic surface of
finite volume $M$ in terms of Eisenstein series. Let us recall
the definition: \begin{defn} Let $\eta$ be a simple closed
geodesic or an infinite geodesic joining $p$ and $q$. A 1-form
$\alpha$ is dual to $\eta$ if for any closed 1-form with compact
support, $\omega$
$$\int_M \omega \wedge \alpha = \int_{\eta}\omega\ .$$\label{a}
Or equivalently: for any closed oriented cycle $c'$, we have

\begin{equation}\label{def}
\int_{c'}\alpha =\eta\ .\ c'.
\end{equation}

\end{defn}

\medskip
 Kudla and Millson construct a meromorphic family of forms on $M$,
called hyperbolic Eisenstein series associated to an oriented simple closed geodesic $c$. Let $\tilde{c}$ be a component of the
inverse image of $c$ in the covering $H\rightarrow M$ and $\Gamma_1$ the stabilizer of $\tilde{c}$ in $\Gamma$.
The hyperbolic Eisenstein series are expressed in Fermi coordinates in the following way for $\Re s >0$:
\begin{equation}\label{eisenhy}
\Omega_c(s,z) =\Omega(s,z) =\frac{1}{k(s)}\sum_{\Gamma_1 \backslash \Gamma }
\gamma^*\frac{dx_2}{(\cosh x_2)^{s+1}},\ \,
k(s):=
\frac{\Gamma\left(\frac{1}{2}\right)\Gamma\left(
\frac{1}{2}+\frac{s}{2}\right)}{\Gamma\left(1+
\frac{s}{2}\right)}.\end{equation}
By applying an element of $SL(2,\mathbb R)$ we may assume that
$\tilde{c}$ is the $y-$axis in $H$ and that $\Gamma_1$ is
generated by $\gamma_1:z\mapsto e^lz$. The Fermi coordinates
$(x_1,x_2)$ associated to $\tilde{c}$ are related to euclidean
polar coordinates by \begin{eqnarray*} r&=&e^{x_1},\\
\sin\theta&=&\frac{1}{\cosh x_2}\, . \end{eqnarray*}
So we can write $\Omega(s,z) = \text{Im}(\Theta(s,z))$ with 
\begin{equation}\label{lien}
 \Theta(s,z)=\frac{1}{k(s)}\sum_{\Gamma_1 \backslash \Gamma }
\gamma^*\left(\frac{y^s}{z|z|^s}dz\right).
\end{equation}
 At the end of
their paper  they do the remark that ``it is also
interesting to consider  the infinite volume case".

\subsection{
The infinite volume case.} We are going to verify that this
definition retains a meaning in the case of a geometrically
finite  hyperbolic  surface of  infinite volume, $M=\Gamma
\backslash H$, $\Gamma$ is a Fuchsian group of the second kind without elliptic elements. We identify $M$ locally with its universal cover $H$. By $d(z,w)$ we denote the hyperbolic distance from $z\in H$ to $w\in H$. For $z_0\in H$ and $\epsilon >0$, by $B(z_0,\epsilon)\subset H$ we denote the radius $\epsilon$ hyperbolic metric ball centered at $z_0$. 
 \begin{prop}\label{kudla} The hyperbolic
Eisenstein series $\Omega(s,z)$ converges  for $\Re s >0$,
uniformly on compact subsets of $H$, is bounded on $M$ and
represents a  $C^\infty$ closed form which is dual to $c$.
Moreover it is an analytic function of $s$ in $\Re s >0$.
\end{prop} The proof in the infinite volume case is as
straightforward as  in the finite volume case (\cite{kmi},
\cite{gerar}), but for the convenience of the
reader we give some details.
\begin{proof}
Recall first that if $K$ is a compact of the fundamental domain $D$ of $\Gamma$ then there exists
$\eta > 0$ such that for all  $z_0 \in K$, the balls  $\left( B(\gamma z_0,\eta)
\right )_{\gamma \in \Gamma_1 \backslash \Gamma}$ are disjoints.


Let us choose  for fundamental domain of $\Gamma_1$,
${\mathcal D}_1 = \{ z \in H : 1 \leq |z| \leq \beta \} $.
After passing to ordinary Euclidean polar coordinates $ (r,
\theta)$,   with $\sigma=\Re s$, with $||\Omega||$ denotes the pointwise norm of $\Omega$, we obtain:
 \begin{eqnarray*} ||\Omega(s,z)|| &\leq&
\frac{1}{|k(s)|} \sum_{\Gamma_1 \backslash \Gamma }
\frac{1}{(\ch x_2(\gamma z))^{\sigma +1}} \\  &\leq&
\frac{1}{|k(s)|} \sum_{\Gamma_1 \backslash \Gamma }
\left(\frac{y}{r}\right)^{\sigma +1 }(\gamma z) \leq
\frac{1}{|k(s)|} \sum_{[\gamma]\in \Gamma_1 \backslash \Gamma, \gamma z\in {\mathcal D}_1} y^{\sigma
+1}(\gamma z). \end{eqnarray*}
 Now $y^s$ is an eigenfunction of all the invariant integral operators on $H$. Let $k(z,z')$ be the point-pair invariant defined by $k(z,z')=1$ or $0$ according as the distance between $z$ and $z'$ is smaller than $\eta$. Then there exists $\Lambda_{\eta}$ independent of $z_0$ so that
$$\int_{B(z_0,\eta)}y^\sigma\frac{dxdy}{y^2}=\int_H k(z_0,z)y^\sigma\frac{dxdy}{y^2}\, $$
and $$\int_{B(z_0,\eta)}y^\sigma\frac{dxdy}{y^2}=\Lambda_{\eta}y(z_0)^\sigma\, .$$
This is a particular formulation of a more general result we will need further (see proposition \ref{mean}).\\
 We denote as in \cite{kmi}:
$R(T_1,T_2)=\{P \in {\mathcal D}_1 : T_1 <x_2(P)<T_2 \}$. So for
$T > 2\eta$ :
$$\frac{1}{|k(s)|} \sum_{\gamma \in \Gamma_1 \backslash
\Gamma ,\gamma z \not \in R(-T,T)}\frac{1}{(\ch(x_2(\gamma
z)))^{\sigma+1}}\leq \frac{1}{|k(s)|}\sum_{[\gamma] \in \Gamma_1 \backslash
\Gamma , \gamma z \in {\mathcal D}_1\backslash  R(-T,T)}y^{\sigma+1}(\gamma z).$$
We need the following:
\begin{lemma}
Let $\gamma \in \Gamma_1\backslash \Gamma $, $z,\ \zeta \in H$ such that for
 $\gamma z \not \in R(-T,T)$, $\gamma\zeta \in B(\gamma z ,
\eta)$ then $\; \gamma \zeta \not \in R(-T+2\eta,T-2\eta)$.
\end{lemma}

\begin{proof}
Let $\pi: H \rightarrow \tilde{c}$ be the orthogonal projection on $\tilde{c}$. As  $\pi$ is 1- lipschitzien, we have for the hyperbolic distance $d(\pi\gamma z,\pi\gamma
\zeta)\leq d(\gamma z, \gamma \zeta)\leq \eta$ . If $x_2(\gamma
z)\geq T$ :
$$T\leq d(\gamma z,\pi \gamma z)\leq d(\gamma z,\gamma \zeta)+d(\gamma
\zeta,\pi \gamma \zeta)+d(\pi \gamma z , \pi \gamma \zeta)\, .$$
 Then
$$T-2\eta \leq d(\gamma \zeta,\pi \gamma \zeta)\, .$$

\end{proof}
Then
\begin{eqnarray*}
\sum_{\gamma \in \Gamma_1 \backslash
\Gamma, \gamma z \not \in R(-T,T) }y^{\sigma+1}(\gamma
z)&=&\frac{1}{\Lambda_\eta}\sum_{\gamma \in \Gamma_1 \backslash
\Gamma,\gamma z \not \in R(-T,T) }\int_{B(\gamma z,\eta)}
y^{\sigma+1}\,\frac{dxdy}{y^2} \\
  &\leq& \frac{1}{\Lambda_\eta}\int_{R^c(-T+2\eta,T-2\eta)}y^{\sigma+1}\,\frac{dxdy}{y^2}
\\
\end{eqnarray*}

where $R^c(-T+2\eta,T-2\eta)$ is the complementary in
${\mathcal D}_1$ of $R(-T+2\eta,T-2\eta)$. Note that if
 $\gamma z \not \in R(-T,T)$ then $y(\gamma z) \leq \displaystyle
\frac{\beta}{\ch T}$, so:

\begin{eqnarray*} \sum_{{\gamma \in \Gamma_1 \backslash
\Gamma}, {\gamma z \not \in R(-T,T) }}y^{\sigma+1}(\gamma z)&\leq&
\frac{\beta}{\Lambda_\eta}\int_0^{\frac{\beta}{\ch(T-2\eta)}}y^{\sigma-1}\,dy
\\ &\leq& \frac{\beta}{\Lambda_\eta \sigma}\left(
\frac{\beta}{\ch(T-2\eta)}\right)^\sigma~~~.\\ \end{eqnarray*}
From this follows the uniform convergence of $\Omega(s,z)$ on
compact subsets of  $H$, uniformly on compact subsets of the
half plane $\Re s>0$.\\ We next show that $\Omega(s,z)$ is bounded on
$D$.  For this we use  a very useful
fundamental lemma (see \cite{heker}, p. 178,
\cite{h2}, p.27):

\begin{prop} For any Fuchsian group $\Gamma$, there exists
${\mathcal C}(q,\Gamma)$, such that for all $z\in H$

$$\sum_{\gamma \in \Gamma}\frac{y(\gamma z)^q}{[1+|\gamma z|]^{2q}} \leq
{\mathcal C}(q,\Gamma).$$
The constant $\mathcal{C}(q,\Gamma)$  depends only of  $q$ and $\Gamma$.

\end{prop}
Let  $z\in H$, there exists a system of representants $S$ of
$\Gamma_1 \backslash \Gamma $ such that for all $\gamma \in S, \;
|\gamma z|\leq \beta $. Then:

\begin{eqnarray*}
\sum_{\Gamma_1 \backslash \Gamma}\frac{y(\gamma
z)^{\sigma+1}}{(1+\beta)^{2(\sigma+1)}}&\leq& \sum_{\Gamma_1 \backslash
\Gamma}\frac{y(\gamma z)^{\sigma+1}}{(1+|\gamma
z|)^{2(\sigma+1)}}\\
&\leq& \sum_{
\Gamma}\frac{y(\gamma z)^{\sigma+1}}{(1+|\gamma z|)^{2(\sigma+1)}}\\
&\leq&{\mathcal C}(\sigma +1,\Gamma)
\end{eqnarray*}
and the result follows.

The fact that $\Omega(s,z)$ is dual to $c$ follows straightly
from the construction of Kudla and Millson.
\end{proof}
\section{ Spectral decomposition and analytic continuation.}
 The aim is to realize the injection
 $H^1_c \rightarrow {\mathcal H}^1$,
 where $H^1_c$ is the first de Rham's cohomology group with
compact support of $M$ and ${\mathcal H}^1$ is the space of
$L^2$ harmonic 1-forms of $M$. Recall that in our context $\dim
{\mathcal H}^1=\infty$ (see \cite{carron}, p. 27).\\
We are going  to prove, as in
\cite{kmi}, the analytic continuation  of the hyperbolic
Eisenstein series. The essential difference  with the finite
volume case is the spectral decomposition of $L^2(M)$.
 
 \subsection{Spectral theory.}
For any non-compact geometrically finite hyperbolic surface $M$,
the essential spectrum of the (positive) Laplacian $\Delta_M$
defined by the hyperbolic metric on $M$ (the Laplacian on
functions) is $[1/4, \infty)$ and this is absolutely continuous.
The discrete spectrum consists of finitely many eigenvalues in
the range $(0,1/4)$. In the finite-volume case one may also have
embedded eigenvalues in the continuous spectrum, but these do not
occur for infinite-volume surfaces. Then if $M$ as infinite
volume, the discrete spectrum of  $\Delta_M$
is finite (possibly empty).
The exponent of convergence
$\delta$ of a Fuchsian group $\Gamma$ is defined to be the
abscissa of convergence of the Dirichlet series:
$$\delta=\inf\{  s> 0, \sum_{T\in \Gamma} e^{-s d(z,Tw)} < \infty\}$$
for some $z,w\in H$ and $d(z,w)$ denotes again the hyperbolic distance from $z\in H$ to $w\in H$.

Let $\Gamma$ be a Fuchsian group of the second kind and
$L({\Gamma})$ be its limit set, then $0<\delta<1$ with $\delta >
1/2$ if $\Gamma$ has parabolic elements. Patterson and Sullivan
showed that $\delta$ is the Hausdorff dimension of the limit set
when $\Gamma$ is geometrically finite. Furthermore, if
$\delta>1/2$, then $\delta(1-\delta)$ is the lowest eigenvalue
of the Laplacian $\Delta_M$. The connection to spectral theory
was later extended to the case $\delta \leq 1/2$ by Patterson.
In this case, the discrete spectrum of $\Delta_M$ is empty and
$\delta$ is the location of the first resonance.
For a detailed account of the spectral theory of infinite area surfaces, we
refer the reader to \cite{borth}.

\subsection{Tensors and automorphic forms.}
This section introduces the notations used in the following subsection \ref{gene} and section \ref{degene}. Let $M$ be a geometrically finite hyperbolic surface. Let $z$ be a local conformal coordinate and $ds^2=\rho|dz|^2$ the Poincar\'e metric. Let $\w_M=T^*M$ be the holomorphic cotangent bundle of $M$ and for any integers $n$ and $m$, let $\E^{r,s}(M,\w_M^n\otimes {\overline \w}_M^m)$ be the vector space of smooth differential forms of type $(r,s)$ on $M$ with values in the line bundle $\w_M^n\otimes {\overline \w}_M^m$ .
 For an integer $q$,  a $q-$form (or $q-$differential) is an element of $T^q=\E^{0,0}(M,\w_M^q)$,
 the space of  tensors of type $q$ on $M$ written locally as $f(z)(dz)^q$.  $M$ may be realized as $\Gamma\backslash H$, where $H$ is the upper half plane and $\Gamma$ a discrete subgroup of $PSL(2,\RR)$.
 The hyperbolic metric on $M$ induces the natural scalar product
$$\left(\phi,\psi\right)=\int_{\Gamma\backslash H} \phi(z)\overline{\psi(z)}y^{2q-2}\, dxdy\, ,$$
on $T^q$. Let $\mathfrak{H}^q$ be the $L^2-$closure of $T^q$ with respect to this scalar product.

 We recall now the link between $q-$forms and automorphic forms of weight\footnote{Some authors
called them of weight $q$ or $-2q$} $2q$, called also $q-$automorphic forms for the following reason. As before we make use of the uniformization theorem.
%

Using notably
the notations in \cite{h1} and \cite{fayreine}, set
$$j_\gamma(z)=\frac{(cz+d)^2}{|cz+d|^2}=\frac{cz+d}{c\bar{z}+d}=\left(\frac{\gamma'z}{|\gamma' z|}\right)^{-1}\, \,
 \gamma=\left(\begin{array}{cc}a&b\\c&d\end{array}\right)\in
\Gamma .$$ Let $\mathcal F_q$ be the space of all
functions $f: H\rightarrow \CC$ with
$$f(\gamma z)=j_\gamma(z)^q f(z),\,\,\,\,\, \gamma\in \Gamma\, ,$$
 and if $\mathcal D=\Gamma\backslash H$ is the fundamental
domain of $\Gamma$, define the Hilbert  space
${\mathcal H}_q=\{f\in {\mathcal F_q}, \<f,f\>_{\mathcal D}=\int_{\mathcal D} |f(z)|^2\, d\mu(z)<\infty\}$ with  areal measure $d\mu(z)=\frac{dxdy}{y^2}$ and the inner product $\<f,g\>=\int_{\mathcal D}f(z)\overline{g(z)}d\mu(z)$. An element in ${\mathcal H}_q$ is called an automorphic form of weight $2q$. ${\mathcal H}_q$ is isometric to ${\mathfrak H}^q$ through the correspondence 
$$I: {\mathfrak H}^q \ni f \mapsto y^q f\,\in {\mathcal H}_q  .$$
 Maass introduced the differential operators
$$\begin{array}{ccc}
 L_q =& (\bar{z}-z)\frac{\partial}{\partial\bar{z}}-q  &: {\mathcal F}_q\rightarrow {\mathcal F}_{q-1}\\ K_q=&(z-\bar{z})\frac{\partial}{\partial z}+q  &: {\mathcal F}_q\rightarrow {\mathcal F}_{q+1}\end{array}$$ 
%

 We  also have:
\begin{eqnarray*}
L_q&=&-2iy^{1+q}\frac{\partial}{\partial\bar{z}}y^{-q}=\overline{K_{-q}},\\
K_q&=&2iy^{1-q}\frac{\partial}{\partial z}y^{q}=\overline{L_{-q}}.
\end{eqnarray*}
We denote
\begin{eqnarray*}
-L_{q+1}K_q&=&-\Delta_{2q}+q(q+1),\\
-K_{q-1}L_q&=&-\Delta_{2q}+q(q-1)
\end{eqnarray*}
with $$\Delta_{2q}=y^2(\frac{\partial^2}{{\partial x}^2}+\frac{\partial^2}{{\partial y}^2})-2iqy\frac{\partial}{\partial x}\, .$$
These second order differential operators are self-adjoint on ${\mathcal H}_q$.

Now, the metric and complex structure determine a covariant derivative


 $$\nabla=\nabla^q\oplus\nabla_q: \E^{0,0}(M,\w_M^q)\longrightarrow \E^{1,0}(M,\w_M^q)\oplus \E^{0,1}(M,\w_M^q) $$
on the line bundle $\w_M^q$. With the identifications  $\E^{1,0}(M,\w_M^q)\cong T^{q+1}$ and $\E^{0,1}(M,\w_M^q)\cong T^{q-1} $, we have 
   
$\nabla^q:T^q \rightarrow  T^{q+1}$,
%
%
$\nabla_q:T^q \rightarrow  T^{q-1}$.

 Under the correspondence $I$, the operators $\nabla_q$, $\nabla^q$ go over to the Maass operators  according to the commutative diagram
$$\begin{array}{ccccc}{\mathfrak H}^{q-1}&\overset{\nabla_q}{\leftarrow}&{\mathfrak H}^q&\overset{\nabla^q}{\rightarrow}&{\mathfrak H}^{q+1}\\\downarrow& &\downarrow& &\downarrow\\\mathcal H_{q-1}&\overset{L_q}{\leftarrow}&\mathcal H_q&\overset{K_q}{\rightarrow}&\mathcal H_{q+1}\end{array}$$
and so are given locally by $\nabla^q=2i\rho^q\partial \rho^{-q}$, $\nabla_q=-2i\rho^{-1}\bar{\partial}$.

 The Laplacians $\Delta_q^+$ and $\Delta_q^{-1}$ on $T^q$ are defined by $\Delta_q^+=-\nabla_{q+1}\nabla^q$, $\Delta_q^-=-\nabla^{q-1}\nabla_q$  and then the isometry $I$ conjugates $\Delta_q^+$ with $-\Delta_{2q}+q(q+1)$ and $\Delta_q^-$ with $-\Delta_{2q}+q(q-1)$.  Thus $\Delta_0=\Delta_0^{\pm}$ is the Laplacian on functions. The operators $\Delta_q^{\pm}$ are non-negative self adjoint.

We are first interested in the case $q=2$. Let $\Delta_{\text{Diff}}$  the
(positive) Laplacian on 1-forms on a geometrically finite
hyperbolic surface, $\Delta_{\text{Diff}}=d\delta + \delta d$,
$\delta=-*d*$ with $*$ the Hodge operator. In the following we
denote  $\Delta_{\text{Diff}}=\Delta$. If $\omega$ is a 1-form in the holomorphic
cotangent bundle,  $\omega=f(z)\, dz$, then we define the image by the isometry $I$,
 $I(w)=I(f\, dz )= yf(z)=\tilde{f}(z)$.  We have the relation $y\Delta(f\,dz)=-(\Delta_2
\tilde{f}) dz$, in other words, with the preceding notations $\Delta =\Delta_1^-$.\\ 

\subsection{Generalized eigenfunctions.}\label{gene} We are going to give the spectral
expansion in eigenforms of $\Delta$; we use \cite{fayreine},
\cite{patt}, \cite{borth}.  With the notations of the section 1, for a finitely generated group of the second
kind, for each cusp and for each funnel of the quotient there is a
corresponding Eisenstein series, this is what we are going to
develop now.

\begin{prop}\label{noyau}
For $\Re s>\delta$, the kernel of the resolvent $G_s(z,w,1)$ for the self-adjoint operator $\Delta_2$ acting on the Hilbert space ${\mathcal H}_1$ of automorphic forms of weight 2, is given by the convergent series
$$G_s(z,w,1)=\sum_{\gamma\in \Gamma}j_\gamma(w)g_s(z,\gamma w, 1)\, $$
with $g_s(z,w,
1)=-\displaystyle\frac{w-\bar{z}}{z-\bar{w}}\frac{\Gamma(s+1)\Gamma(s-1)}{4\pi
\Gamma(2s)}\cosh^{-2s}(d(z,w)/2)F(s+1,s-1,s;\cosh^{-2}(d(z,w)/2))\, $ and $F$ is the
Gauss hypergeometric function.

\end{prop}

For the funnel case, we identify $z'$ with the standard
coordinates $({\rho'},t')$ in the funnel $F_j$, and we denote
\begin{equation}\label{fun}
E^f_{j,1}(s,z,t')=\lim_{\rho '\to 0} {\rho'}^{-s}G_s(z,z',1)\, ,
\end{equation}
for $j=1,...,n_f$. In the cusp $C_j$, with standard coordinates
$z'=({\rho'},t')$, we set
\begin{equation}\label{cusp}
E^c_{j,1}(s,z)=\lim_{\rho '\to 0} {\rho'}^{1-s}G_s(z,z',1)\, ,
\end{equation}
for $j=1,...,n_c$.

 Let
 $$P(z,\zeta)=\Im(z)/{|z-\zeta|^2} $$
 where $z\in H$ and $\zeta\in \RR$, be the Poisson kernel. For
$b\in O(\Gamma)={\mathbb R }\cup \{\infty\}-
L({\Gamma})$ define the Eisenstein series (\cite{patt},
\cite{els})
$$E_b(z,s,k)=\sum_{\gamma \in\Gamma} j(\gamma,z)^k P(\gamma(z), b)^s(\gamma(z),b)^k\, ,$$
where $j(\gamma,z)=\gamma'(z)/|\gamma'(z)|$ and
$(z,b)=(\bar{z}-b)/(z-b)$. This converges uniformly on compact subsets of $H$ if $\Re(s)>\delta$.

\begin{prop}
The series $E_b(z,s,k)$ can be continued to the whole complex plane as a meromorphic function in $s$.
\end{prop}

One verifies that
 $-\Delta_{2k} E_b(z,s,k)=s(1-s)E_b(z,s,k)$.

\begin{rem}
Thus if $\delta<1/2$, $E_b(z,s,k)$ is analytic in a neighbourhood of $\Re(s)=1/2$.
\end{rem}
For the standard funnel $F_l$ which corresponds to the region $\Re
z\geq 0$ in the model $C_l=\Gamma_l\backslash H$, we have (see
\cite{fayreine} p.200): \begin{eqnarray*}
E^f_{l,1}(s,z,x')&=&\lim_{z'\to x'}(\Im z')^{-s}G_s(z,z',1)\\
&=&-\frac{4^s}{4\pi}\frac{\Gamma(s+1)\Gamma(s-1)}{\Gamma(2s)}E_{x'}(z,s,1).
\end{eqnarray*}
We write then
$${\mathcal E}_{f_l}(s,z,x')=(1-2s)\frac{E^f_{l,1}(s,z,x')}{y}\, dz .$$

\begin{rem}\label{rappell}
Recall  the definition of the classical Eisenstein series.
The stabilizer of a cusp $\A$ is an infinite cyclic group generated by a parabolic motion,
$$\Gamma_{\A}=\{ \gamma \in \Gamma : \gamma \A=\A\}=\< \gamma_{\A} \>\ ,$$
say. There exists $\sigma_{\A} \in SL_2(\mathbb{R})$, called  a scaling matrix of the cusp $\A$,  such that
$\sigma_{\A}\infty=\A,\ \sigma_{\A}^{-1}\gamma_{\A} \sigma_{\A}=\left( \begin{array}{cc}1&1\\ 0 &1\end{array}\right) $.
$\sigma_{\A}$ is determinated up to composition with a translation from the right. The Eisenstein series for the cusp $\A$ is then defined by:
$$E_{\A}(z,s)=\sum_{\Gamma_{\A}\backslash \Gamma}y(\sigma_{\A}^{-1}\gamma z)^s\ ,$$
where $s$ is a complex variable with $\Re s >\delta$.
\begin{defn}\label{weight}
In a similar way we define an Eisenstein series of weight $2q$ associated to a cusp $\A$  as the automorphic form of weight $2q$, for $\Re s >\delta$:
$$\ E_{\A,q}(s,z)=  \sum_{\gamma \in \Gamma_{\A}\backslash \Gamma}y(\sigma_{\A}^{-1} \gamma z)^{s}j_{\sigma_{\A}^{-1} \gamma}(z)^{-q}
 =  \sum_{\gamma \in \Gamma_{\A}\backslash \Gamma}y(\sigma_{\A}^{-1} \gamma z)^{s} \left(\frac{(\sigma_{\A}^{-1} \gamma)'z}{|(\sigma_{\A}^{-1} \gamma)'z |}\right)^q\ .$$
It can be continued to the whole complex plane as a meromorphic function in $s$.
\end{defn}

We call a horocyclic Eisenstein series the 1-form corresponding to the Eisenstein series of weight 2 associated to a cusp $\A$, 
$\ E_{\A,1}$ and  defined for $\Re s >1$, by:

\begin{equation}\label{eh}
{\mathcal E}_{\A}(s,z)=\sum_{\gamma \in \Gamma_{\A}\backslash \Gamma}y(\sigma_{\A}^{-1} \gamma z)^{s-1} \ d(\sigma_{\A}^{-1} \gamma z)
 =\frac{E_{\A,1}(s,z)}{y}\ dz\ .
\end{equation}
\end{rem}

 We now verify that it corresponds to the defining
formula (\ref{cusp}).\\ For the standard cusp, we write
$$G_s(z,z',1)=\sum_{\Gamma_\infty\backslash \Gamma}\left(\frac{c\bar{z}+d}{cz+d}\right)G_s^{\Gamma_\infty}(\gamma z, z',1)\, ,$$
where $ G_s^{\Gamma_\infty}( z, z',1)$ is the resolvent
kernel of the standard cusp for automorphic forms of weight 2.
We use then \cite{fayreine} p. 155 (38), p.177 for $\Im z'>\Im
\gamma z$, p. 172
(see also \cite{borth} p.72, p.102) to conclude that
$$\lim_{y'\to \infty}{y'}^{s-1}G_s(z,z',1)=\sum_{\Gamma_\infty\backslash \Gamma}\left(\frac{c\bar{z}+d}{cz+d}\right)\frac{(\Im \gamma z)^s}{1-2s}=\frac{1}{1-2s}E_{\infty,1}(s,z)\, .$$

Recall the decomposition (see section 1) $M=K\cup_{j=1}^{n_c}C_j\cup_{j=1}^{n_f} F_j$ and denote ${\mathcal E}_{c_j}$ the horocyclic Eisenstein series associated to the cusp $C_j$.

With the preceding notations we then  have (see for example
\cite{fayreine}, \cite{patt})
\begin{prop}
  For $w=f(z)\,dz$ square
integrable,  \begin{eqnarray*}
w(z)&=&\sum_{i=1}^m (w)_{\lambda_i}(z)+ \frac{1}{4\pi i}
\sum_{j=1}^{n_c}\int_{-\infty}^{+\infty}\<w, {\mathcal
E}_{c_j}(1/2 +it, .)\>{\mathcal E}_{c_j}(1/2 +it, z)\, dt\, +\\
& &\frac{1}{4\pi i}
\sum_{j=1}^{n_f}\int_{-\infty}^{+\infty}\left[
\int_1^{\lambda_f^2}\<w, {\mathcal E}_{f_j}(1/2 +it,
.,b)\>{\mathcal E}_{f_j}(1/2 +it, z,b)\, db\right]\, dt;
\end{eqnarray*}
where the first sum in the right member is the projection of $w$ on the discrete spectrum.
\end{prop}
 
 \begin{rem}
 
  One can easily deduce the formula for an arbitrary
square integrable 1-form 
\begin{eqnarray*}
 & &\Omega(z)=f\, dz +g\, d\bar{z}=\sum_{i=1}^m (\Omega)_{\lambda_i}(z)+\\
  & &\frac{1}{4\pi i}
\sum_{j=1}^{n_c}\int_{-\infty}^{+\infty}\<\Omega, {\mathcal E}_{c_j}(1/2 +it, .)\>{\mathcal E}_{c_j}(1/2 +it, z)\, 
+ \<\Omega, {\mathcal E}_{c_j}(1/2
+it, .)_{-1}\>{\mathcal E}_{c_j}(1/2 +it, z)_{-1}\, dt \\
& &+\frac{1}{4\pi i} \sum_{j=1}^{n_f}\int_{-\infty}^{+\infty}
\int_1^{\lambda_f^2}\<\Omega, {\mathcal E}_{f_j}(1/2 +it,
.,b)\>{\mathcal E}_{f_j}(1/2 +it, z,b)\, db\, dt \\
 & &+\frac{1}{4\pi i} \sum_{j=1}^{n_f}\int_{-\infty}^{+\infty}
\int_1^{\lambda_f^2}
\<\Omega,
{\mathcal E}_{f_j}(1/2 +it, .,b)_{-1}\>{\mathcal
E}_{f_j}(1/2 +it, z,b)_{-1}\, db\, dt 
\end{eqnarray*}
where, with obvious notations, ${\mathcal
E}_{-1}=\overline{\mathcal
E}=\displaystyle\frac{\overline{E}}{y}d\bar{z}$. To simplify we
will write
$$\Omega(z)=
(\Omega)_{\lambda_i}(z)+ \frac{1}{4\pi i}
\int_{-\infty}^{+\infty}\<\Omega, {\mathcal E}_{c_j}(1/2
+it, .)_{\pm}\>{\mathcal E}_{c_j}(1/2 +it, z)_{\pm}\, dt\,
 +$$
 $$\frac{1}{4\pi i}\int_{-\infty}^{+\infty}\left[ \int_1^{\lambda_f^2}\<\Omega,
{\mathcal E}_{f_k}(1/2 +it, .,b)_{\pm}\>{\mathcal
E}_{f_k}(1/2 +it, z,b)_{\pm}\, db\right]\, dt\, .$$
\end{rem}
 



\subsection{Harmonic dual form.}

We are now going  to see
\begin{prop}
The hyperbolic Eisenstein series $\Omega_c$
are square integrable.
\end{prop}
 \begin{proof}
We consider a fundamental domain $D$ contained in $\{z, 1\leq
|z|\leq e^l\}$ in which the segment $(i, ie^l)$ represents the
geodesic $c$. We denote $C_\lambda=\{z\in D, d(z,c)=\lambda\}$ and
$F_\lambda=\{z\in D, d(z,c)\geq\lambda\}$. Without loss of
generality we can suppose that there is only one funnel on $M$
and no cusps. Let $V_\lambda$ the volume of $F_\lambda -
F_{\lambda +1}$ there exists a constant $c_1$ such that
$V_\lambda\geq c_1(\sinh(\lambda+1) - \sinh(\lambda))$. For $\Re s=\sigma>0$,
$||\Omega_c(s,z)|| =||\Omega(s,z)|| \leq\frac{1}{|k(s)|}\sum_{\Gamma_1 \backslash \Gamma }
\frac{1}{(\cosh x_2(\gamma z))^{\sigma+1}}$.
Let $\eta(z)= \sum_{\Gamma_1 \backslash \Gamma }
\frac{1}{(\cosh x_2(\gamma z))^{\sigma+1}}$. We know from the section 2 that there exists a constant $K>0$ such that $\forall z\in H$, $\eta(z)\leq  K$.\\
We have \begin{eqnarray*}  \int_D ||\Omega_c(s,z)||^2 d\mu(z)&
\leq& \frac{1}{|k(s)|^2}\int_D \eta^2(z)d\mu(z)\\ &\leq &
\frac{1}{|k(s)|^2}\int_{1\leq x_1 \leq e^l\, ,-\infty
<x_2<+\infty} \eta(z)\frac{1}{(\cosh x_2(z))^{\sigma+1}}\cosh
x_2 \, dx_1\,dx_2\\   &\leq & \frac{K}{|k(s)|^2}\int_{1\leq x_1
\leq e^l\, ,-\infty <x_2<+\infty} \frac{1}{(\cosh
x_2(z))^{\sigma+1}}\cosh x_2 \, dx_1\,dx_2\,  .\end{eqnarray*}
The last integral is $
\frac{\Gamma\left(
\frac{1}{2}\right)\Gamma\left(
\frac{\sigma}{2}\right)}{\Gamma\left(\frac{1}{2}+
\frac{\sigma}{2}\right)}(e^l-1)$ and the result follows.
\end{proof}
 
As in \cite{kmi} we verify:
$$\Delta(\Omega(s,z)) +s(s+1)\Omega(s,z)=s(s+1)\Omega(s+2,z)\, .$$
This formula has a consequence that, for fixed $s$ with $\Re s>0$, the function
$\Delta^k(\Omega(s,z))$ is again square integrable for any $k>0$.\\
Set $\Re s >0$, with our convention of notations
$$\Omega(s,z)=\Omega_0(z) + a_i(s)\phi_i(z) +
\frac{1}{4\pi i}
\int_{-\infty}^{+\infty}h^c_{\pm}(s,t) {\mathcal E}_c(1/2 +it, z)_{\pm}\, dt $$
\begin{equation}\label{star}
 + \frac{1}{4\pi i}
\int_{-\infty}^{+\infty}\left[ \int_1^{\lambda_f^2}H^f_{\pm}(s,t,b){\mathcal E}_f(1/2 +it, z,b)_{\pm}\, db\right]\, dt \,  . 
\end{equation}
 We obtain
 $$H(s,t,b)[1/4 +t^2+s(s+1)]=s(s+1)H(s+2,t,b)\, ,$$
 where $H$ corresponds to any $H^f_{\pm} $.
 From this we get a continuation of $H$ to the region $\Re
s>-1/2$ and we note that  for all $t$ and all $b$ we have
$H(0,t,b)=0$.\\  Moreover for $\Re s>-1/2$, $\Re(s+2)>0$ and we
may substitute in (\ref{star}) to obtain a continuation  of
$\Omega(s,z)$ to $\Re s> -1/2$. Thus we have proved the
following theorem
\begin{thm}\label{dual}  $\Omega(s,z)$ has a meromorphic
continuation to $\Re s>-1/2$ with $s=0$ a regular  point and
$\Omega(0,z)$ is a harmonic form which is  dual to $c$.
\end{thm}
\begin{rem}
1) Another way to see this:\\
 write $\Omega(s,z)=  (\Delta+s(s+1))^{-1}(s(s+1)\Omega(s+2,z)) $
 and use the meromorphic continuation of the resolvent (see
for example \cite{borth}, \cite{ris}).\\ 2) With an analogue
study of \cite{kmi} (see  also \cite{kramer}) we can obtain a
total description of the singularities of the hyperbolic Eisenstein series.  \end{rem}
 
    \section{ The case of an infinite geodesic joining two points.}
  Without loss of generality we suppose the two cusps $p$ and $q$
to be 0 and $\infty$ respectively and, as the lift of the
geodesic, we take the imaginary axis.  Let $\eta$ be the
infinite geodesic $]p,q[$,    can we do the same construction as
Kudla and Millson?    As in the finite volume case, the problem
reduces to study the following series for $\Re s>1$:
\begin{equation}\label{infinity}
 \hat{\eta}^s(z)= \frac{1}{k(s-1)}\sum_{\gamma \in \Gamma}\gamma* \left[\left(
\frac{y}{|z|}\right)^{s-1}\Im(z^{-1}dz)\right]= \Im (\theta^s(z))\ ,
\end{equation}
where

$$\theta^s(z)=\frac{1}{k(s-1)}\sum_{\gamma \in \Gamma}\gamma* \left[\left(
\frac{y}{|z|}\right)^{s-1}\frac{dz}{z}\right]\ ,$$
 and $k(s-1)=\displaystyle \frac{\Gamma(1/2)\Gamma(s/2)}{\Gamma(1/2+s/2)}$.
 \subsection{Some useful estimates.}
 
 As usual we can suppose $\Gamma_\infty=\< z\mapsto z+1 \>$ to be
the stabilizer of $\infty$ in $\Gamma$ and the stabilizer of $0$,
 $\Gamma_0$ is then generated by $z \mapsto \frac{z}{-c_0^2 z+1}$ (for
some non zero constant  $c_0$).\\ First of all we note that,
contrary to the finite volume case, we have 
\begin{lemma}
The series
$\sum_{\gamma\in
\Gamma_\infty\backslash\Gamma}  \Im(\gamma z)$  is convergent (see
proposition \ref{noyau} and formula \ref{cusp}).
\end{lemma}
Another way to see this is ``by hand":
 We know
that for $\Re s>\delta$, $\sum_{T\in\Gamma}e^{-sd(i,Tz)}$
converges, moreover there exists a constant $C>0$ such that
$\sum_{\gamma\in \Gamma_\infty\backslash\Gamma}  \Im^s(\gamma
z)\leq C \sum_{T\in\Gamma}e^{-sd(i,Tz)}$, as in our case
$\delta<1$, we have the result.\\ As $\sum_{\gamma\in
\Gamma_\infty\backslash\Gamma}  \left|\Im^s(\gamma
z)\frac{c\bar{z}+d}{cz+d}\right|=\sum_{\gamma\in
\Gamma_\infty\backslash\Gamma}  \Im^{\Re s}(\gamma z)$ we also
deduce the convergence of the series representing $E_{\infty,1}(1,z)$ (\ref{rappell}).
 With the notations of the recall 3.1 and $\rho$ the standard coordinate for the cusp $\A$, we re-write the results of
\cite{borth}(p.110) in the following way: $E_{\A}(s,.)=\rho^{-s}(1-\chi_0(\rho))+0(\rho_f^s \rho_c^{s-1})$, $\rho$ is decomposed as $\rho_f \rho_c$ with $\rho_f=\rho$ in the funnels and $\rho_c=\rho$ in the cusps and we define $\chi_0\in {\mathcal C}_0^\infty(X)$ such that 
$$\chi_0=\left\{ \begin{array}{cc}1,& r\leq 0\\0,& r\geq 1\end{array}\right ..$$
\begin{lemma}\label{prop1} We have the following asymptotic behaviors for $ \Re
s>\delta$:\\ 
 1) in a  funnel, for all cusp $\A$,
$E_{\A}(s,z)$ is square integrable;\\ 2) at $\A=\infty$,
$E_\infty(s,z)- y^s=O(y^{1-s})$ and $E_0(s,z)=O(y^{1-s})$;\\  3) near $\A=0$,
$E_0(s,z)-y^s/(c_0^2|z|^2)^s=O(y^{1-s}/(c_0^2|z|^2)^{s-1})$ and $E_\infty(s,z)=O(y^{1-s}/(c_0^2|z|^2)^{s-1})$.
\end{lemma}

\subsection{Convergence of the Hyperbolic  Eisenstein series and
analytic continuation.}

The computations to prove the convergence of
(\ref{infinity}) are easily adapted from the finite volume case. For the
convenience of the reader we recall the essential points.\\  We
have $|| \sum_{\gamma \in
\Gamma}\gamma*\left(\frac{y}{|z|}\right)^{s-1}\Im(z^{-1}dz)||\leq
\sum_{\gamma \in \Gamma}\left( \frac{y}{|z|}\right) ^{\sigma}
(\gamma z)~~~ ,$ where $\sigma = \Re s >1$ and if we denote by
${\mathcal S}= \displaystyle \sum_{\gamma \in \Gamma}\left(
\frac{y}{|z|}\right) ^{\sigma} (\gamma z)$, we have
$${\mathcal S}=\sum_{\gamma \in \Gamma_\infty \backslash
\Gamma}y^{\sigma}(\gamma z) \sum_{n \in \mathbb{Z}}\frac{1}{|\gamma z
+n|^{\sigma}}~~~.$$
Let $S_z$ be a system of representatives of
$\Gamma_{\infty}\backslash \Gamma$ such that $|\Re \gamma z| \leq 1/2$, then
$$||{\mathcal S}||\leq \sum_{\gamma \in S_z}\frac{y^\sigma(\gamma z)}{|\gamma
z|^\sigma} + 2\sum_{\gamma \in S_z}y^\sigma(\gamma z) \sum_{n=1}^\infty
\frac{1}{(n-1/2)^\sigma}\, .$$

We have
\begin{eqnarray*}
\sum_{\gamma \in S_z}\frac{y^\sigma(\gamma z)}{|\gamma
z|^\sigma}&=&\sum_{\Gamma_0\backslash S_z}\frac{y^\sigma(\gamma z)}{|\gamma
z|^\sigma}\sum_{n\in \mathbb{Z}}\frac{1}{|-nc_0^2 \gamma z+1|^\sigma}\\
  &=& \sum_{\gamma \in \Gamma_0 \backslash S_z}
\frac{y^\sigma(\gamma z)}{|\gamma
z|^\sigma}+ \sum_{\gamma \in \Gamma_0\backslash S_z}
\frac{y^\sigma(\gamma z)}{|\gamma z|^\sigma}\sum_{n\in \mathbb{Z}^*}\frac{1}{|nc_0^2|^\sigma[(x(\gamma z)-1/nc_0^2)^2 +y^2( \gamma z)]^{\sigma/2}}.
\end{eqnarray*}

For $K$ a compact set in $ H$ there exists  $m$ in
$ H$ such that
 $$\forall z \in K,\, \forall \gamma \in \Gamma_0\backslash
S_z,\,|\gamma z| \geq |m|\,\, \mbox{and}\,\, \Im \gamma z\geq \Im
m\, .$$
So
 \begin{eqnarray*} \sum_{\gamma
\in S_z} \frac{y^\sigma (\gamma z)}{|\gamma z|^\sigma }&\leq&
\sum_{\gamma \in \Gamma_0 \backslash S_z} \frac{y^\sigma (\gamma
z)}{|m|^\sigma } +\sum_{\gamma \in \Gamma_0 \backslash
S_z}\frac{y^\sigma(\gamma z) }{|m|^\sigma}\sum_{n\in
\mathbb{Z}^*}\frac{1}{|nc_0^2|^\sigma (\Im m)^\sigma }\\  &\leq&
\frac{1}{|m|^\sigma }\sum_{\Gamma_\infty \backslash
\Gamma}y^\sigma (\gamma z)+ 2\sum_{n\in
\mathbb{N}^*}\frac{1}{(nc_0^2)^\sigma       \bigskip
}\frac{1}{|m|^\sigma (\Im m )^\sigma  }\sum_{\Gamma_\infty
\backslash \Gamma}y^\sigma(\gamma z)~~~; \end{eqnarray*}
and finally, for all $z$ in $K$
 $$\aligned
||\mathcal{S}|| \leq &\ \frac{1}{|m|^\sigma }\sum_{\Gamma_\infty \backslash
\Gamma}y^\sigma(\gamma z)+ 2\sum_{n\in \mathbb{N}^*}\frac{1}{(nc_0^2)^\sigma}
\frac{1}{|m|^\sigma (\Im m )^\sigma  }\sum_{\Gamma_\infty \backslash
\Gamma}y^\sigma(\gamma z) +\\
& \ 2\sum_{\Gamma_\infty \backslash \Gamma} y^\sigma
(\gamma z) \sum_{n=1}^\infty \frac{1}{(n-1/2)^\sigma},
\endaligned$$
the uniform convergence on all compact of $H$ and all compact of
 $\Re s > 1$.

 From this ultimate inequality we conclude that $\theta^s$ is square integrable in the funnels as
 the Eisenstein series $\mathcal E_\infty$.
To conclude, we have the following theorem:
\begin{thm}\label{3.1}
For $\Re s>1$, the Eisenstein series associated to the geodesic
$\eta = (p,q)$ converges uniformly on all compact sets. It represents a ${\mathcal C}^\infty$ closed
form which is  dual to
 $\eta$. For $\Re s>1$ it satisfies the differential functional equation:

$$\Delta \hat{\eta}^s=s(1-s)[\hat{\eta}^s-\hat{\eta}^{s+2}].$$
\end{thm}
 
 Now we want to prove the analytic continuation of $\hat{\eta}^s$
at $s=1$. For this, first of all, we are going to show that
$\theta^s(z)-1/i({\mathcal E}_\infty(1,z)-{\mathcal E}_0(1,z))$ is
square integrable. As we have shown that $\theta^s$ is square integrable in the funnels, what we have to do is to investigate the
Fourier expansion  of $\theta^s$, at each inequivalent cusps:
$0$ and $\infty$,  and to show that $ ||\theta^s||$ is
bounded at the cusps. As in the case of  finite volume case we have (\cite{mpi}):

\begin{prop}\label{prop2}
At $\infty$
$$\theta^s(z)=(\frac{1}{i} +O(1/y))\ dz\ ,$$
and at $0$
$$\theta^s(z)=(-\frac{1}{i c_0^2 z^2} +O(1/y))\ dz\ .$$
\end{prop}

By proposition \ref{prop1} and lemma \ref{prop2}, we conclude:

\begin{prop}\label{prop3}
The 1-forms $\theta^s(z)-1/i({\mathcal E}_\infty(1,z)-{\mathcal E}_0(1,z))$ and  $\hat{\eta}^s(z)+\Re({\mathcal E}_\infty(1,z)-{\mathcal E}_0(1,z))$ are square integrable.
\end{prop}

Finally as in \cite{mpi}:
\begin{thm}\label{infinite}
The 1-form $\hat{\eta}^s$ has a meromorphic continuation to $\Re s >1/2$, with $s=1$ a regular point and $\hat{\eta}$ is a harmonic  form which is dual to $\eta$.
\end{thm}

\section{Degenerations.}\label{degene}
\subsection{Background material and the main results.}
A family of degenerating hyperbolic surfaces consists of a
manifold $M$ and a family $(g_l)_{l>0}$ of Riemannian metrics on
$M$ that satisfy the following assumptions: $M$ is an oriented
surface of negative  Euler characteristic and the metrics $g_l$
are hyperbolic, chosen in such a  way that there are finitely
many closed curves $c_i$, geodesic with respect to all metrics,
with the length $ l_i$ of each curve converging to 0 as $l$
decreases. On the complement of the distinguished curves, the
sequence of metrics is required   to converge to a hyperbolic
metric. More precisely there are finitely many disjoint open
subsets $C_i\subset M$ that are   diffeomorphic to cylinders
$F_i\times J_i$ where $J_i\subset\mathbb R$ is a   neighborhood of
$0$. The complement of $\bigcup_i C_i$ is relatively compact.
The restriction of each metric $g_l$ to $C_i=F_i\times J_i$ is a
product metric
    $$ (x,a)\longmapsto(l_i^2+a^2)dx^2+(l_i^2+a^2)^{-1}da^2 $$
   and $l_i\to 0$ as $l\to0$ (the curves $F_i\times\{0\}\subset
C_i$ are closed geodesics of length $l_i$ with respect to
$g_l$). Let $M_{l}$ denote the surface $M$ equipped with the
metric $g_l$ if $l>0$ and let $M_0=M\backslash \bigcup_i c_i$
carry
 the limit metric $\lim_{l\to0}g_{l}$. Note that $M_0$ is a
complete hyperbolic surface by definition, which contains a pair
of cusps for each $i$.\\
Here we consider a  family of surfaces $S_l=\Gamma_l\backslash H$
degenerating to the surface $S$ with only one geodesic $c_l$
being pinched, $\Gamma_l$ containing the transformation
$\sigma_l(z)=e^lz$ corresponding to $c_l$. Let $K_l$ be $S_l$
minus $C_l$ the standard collar for $c_l$. There exist
homeomorphisms $f_l$ from $S_l\backslash c_l$ to $S$, with
$f_l$ tending to isometries $C^2$-uniformly on the compact core
$K_l \subset S_l$; define $\pi_l=f_l^{-1}$. Suppose that $p$ is one of the two cusps of
$S$ arising from pinching $c_l$. Let
$S_0=\Gamma\backslash H$ be the component of $S$ containing
$p$ and conjugate $\Gamma$ to represent the cusp by the
translation $w\mapsto w+1$, in the following $p=\infty$.\\
 Let for $\Re s>1$,
$\alpha_l(s,z)=\sum_{\gamma \in \<\sigma_l\>\backslash\Gamma_l}\gamma* \left[\left(
\frac{y}{|z|}\right)^{s-1}\Im(z^{-1}dz)\right]$ such that the
hyperbolic Eisenstein series $\Omega_{c_l}=\Omega_l$ is related
by $\Omega_l(s,z)=\frac{1}{k(s)}\alpha_l(s+1,z)$. Without loss of generality we
suppose $S_l$ having only one funnel $F_1$. With the
notations of the beginning, $S_l=K\cup(C_1\cup...\cup
C_{n_c})\cup F_1$ and $c_l$ is the one geodesic of the boundary
of the compact core $K$. We consider the specific case of $p$,
the limit of the right side of the $c_l$-collar, contained in
$S_l\backslash F_1$ .
\begin{thm}\label{deg} Let $\Re s>1$, the family of 1-forms
$\frac{1}{l^s}\alpha_l(s,\pi_l(.))$ converges uniformly on
compact subsets  of $S_0$ to $\Im{\mathcal E}_\infty(s,.)$.
\end{thm}
It is a particular case of the  theorem \ref{auto} below.

\smallskip
The sketch of the proof of this theorem follows these of 
the finite volume case (\cite{preprint}, \cite{w2}, see also 
\cite{jorg}). First of all  we recall
some material and results.

The following lemma can be found for example in \cite{borth}. The neighborhood of points within distance $a$ of a geodesic
$\gamma$, where $d(z,\gamma)$ is the hyperbolic distance from $z$ to $\gamma$,
$$G_a=\{z\in K, d(z,\gamma)\leq a\}\, ,$$
is isometric for small $a$ to a half-collar $[0,a]\times S^1$,
$ds^2=dr^2 +l^2 \cosh^2 r\, d\theta^2$. 
\begin{lemma}
Suppose that $\gamma$ is a simple closed geodesic of length
$l(\gamma)$ on a geometrically finite hyperbolic surface $M$.
Then $\gamma$ has a collar neighborhood of half-width $d$, such
that

$$\sinh(d)=\frac{1}{\sinh(l(\gamma)/2)}\, .$$
 As a consequence,
if $\eta$ is any other closed geodesic intersecting $\gamma$
transversally (still assuming $\gamma$ is simple), then the
lengths of the two geodesics satisfy the inequality

$$\sinh(l(\eta)/2)\geq \frac{1}{\sinh(l(\gamma)/2)}\, .$$

\end{lemma}
\begin{lemma}
Let $\gamma$ be a simple closed geodesic of length $l$ on a
complete hyperbolic surface $M$. If $\alpha$ is a simple closed
geodesic that does not intersect $\gamma$  then, where $d(\gamma,\alpha)$ is the hyperbolic distance of $\gamma$ to $\alpha$
$$\cosh d(\gamma, \alpha)\geq \coth(l/2)\, .$$
\end{lemma}
A standard collar for a length $l$ geodesic is a cylinder
isometric to ${\mathcal C}\backslash \<z\mapsto e^l z\>$ with
${\mathcal C}=\{ z=re^{i\theta}, 1\leq r\leq e^l,
l<\theta<\pi-l\}\subset H$ with the restriction of the
hyperbolic metric, and $\<z\mapsto e^l z\>$ the cyclic
group generated by the transformation $z\mapsto e^l z$. There is a constant $k_0$ (the
short geodesic constant) such that each closed geodesic on $S_l$
of length at most $k_0$ has a neighborhood isometric to the
standard collar and each cusp for $S_l$ has a neighborhood
isometric to the standard cusp; furthermore, the collars for
short geodesics and the cusp regions are all mutually disjoints.

Now to study the right
side of the $c_l$-collar let $w=\frac{1}{l} \log z$, with the principal branch $z\in H$,
and conjugate $\Gamma_l$ by the map $w$ to obtain
${\tilde\Gamma}_l$ acting on ${\mathcal S}_l=\{w, 0<\Im
w<\pi/l\}$. The hyperbolic metric on ${\mathcal S}_l$ is
$ds_l^2=\left(\frac{l|dw|}{\sin (l\Im w)}\right)^2$, which tends
uniformly on compact subsets to $\left(\frac{|dw|}{\Im
w}\right)^2$.  ${\tilde\Gamma}_l$ is a (non M\"obius) group of
deck transformations acting on ${\mathcal S}_l$; the quotient
${\tilde\Gamma}_l\backslash{\mathcal S}_l$ is $S_l$. Let
$\hat{f}_l$ be the restriction  of $f_l$ to the component ${\mathcal S}^{(r)}_l$ of
$S_l\backslash c_l$ containing the right half-collar for $c_l$.
Let $F_l$ be a lift of $\hat{f}_l$ to the universal covers
${\mathcal A}_l$ and $H$, where  ${\mathcal A}_l$ is the simply
connected component of $H\backslash \pi^{-1}(c_l)$ which
contains the standard right collar $\{z=re^{i\theta}, 1\leq
r\leq e^l,l<\theta<\pi/2\}$. More precisely \cite{preprint}(p.350), \cite{w2}:
\begin{lemma}
The simply
connected component  ${\mathcal A}_l$ contains $\{z=re^{i\theta}, 1\leq
r\leq e^l,lc(l)<\theta<\pi/2\}$ where $c(l) \to 0,\, l\to 0$.
\end{lemma}
Start with the standard $\Gamma$ fundamental domain
${\mathcal F}=\{w\in H, 0\leq \Re w<1, \Im w\geq \Im A(w), \forall A\in
\Gamma\}$.
Set $D_l=F_l^{-1}(\mathcal F)$,
then $D_l$ is a fundamental domain of $S_l$.
Divide the cosets of $\< \sigma_l \>\backslash (\Gamma_l - \< \sigma_l \>)$ into two classes $D= \{[A],
A\in\Gamma_l, \inf \Re A(D_l)>0\}$ and
$G=\{[A], A\in\Gamma_l, \sup \Re
A(D_l)<0\} $.

Then  $\hat{f}_l$ has a lift $\tilde{f}_l$, a homeomorphism from
a subdomain of ${\mathcal S}_l$  to $H$: $\tilde{f}_l=F_l\circ w^{-1}: w({\mathcal A}_l)\rightarrow H$. The homeomorphism $\tilde{f}_l$
induces a group homomorphism $\rho_l: \Gamma\rightarrow
\tilde{\Gamma}_l$,  $A\mapsto \tilde{f}_l
^{-1}A\tilde{f}_l$, $A\in\Gamma$. We call $\rho_l(A)\in
\tilde{\Gamma}_l$ the element corresponding to $A\in\Gamma$.
Now by our normalizations for $\tilde{\Gamma}_l$ and $\Gamma$
the translation $w\mapsto w+1$ corresponds to itself. If we
specify the further normalization $ \tilde{f}_l(i)=i$ then the
lifts $\tilde{f}_l$ are uniquely  determined and then we have (see e.g. \cite{w2} p.107):
\begin{lemma}\label{lem}
The $\tilde{f}_l$ tend uniformly on compact
subsets  to the  identity, and thus for $A\in\Gamma$, the
corresponding elements $\rho_l(A)$ tend uniformly on compact
subsets to $A$.
\end{lemma}
 For $q\in \NN$, we associate to the pinching geodesic $c_l$, the $q$-form defined for $\Re s >1$ by:
$$A_{l,q}(s,z)=\sum_{\<\sigma_l\>\backslash \Gamma_l} \left(\frac{\gamma'(z)}{\gamma(z)}\right)^q \sin^{s-q}\theta(\gamma z)\, dz^q\, .$$
 Divide the cosets $\< z\mapsto z+1\>\backslash (\tilde{\Gamma}_l
-\< z\mapsto z+1\>)$ into two classes, the left and the right: for
${\mathcal F}_l=\tilde{f}_l^{-1}({\mathcal F})$, $L=w G w^{-1}=\{[A],
A\in\tilde{\Gamma}_l, \inf \Im A({\mathcal F}_l)>\pi/2l\}$ and
$R=w D w^{-1}=\{[A], A\in\tilde{\Gamma}_l, \sup \Im
A({\mathcal F}_l)<\pi/2l\}$ (the line $\{ \Im w=\pi/2l\}$ is a
lift of $c_l$, and we write $[A]$ for the $\<z\mapsto z+1\> $ coset of
$A$).
In particular the cosets $\< z\mapsto z+1\>\backslash (\Gamma-\<z\mapsto
z+1\> ) $ correspond to the right cosets of $\tilde{\Gamma}_l$: $\{[\rho_l(A)], A\in \Gamma, \<w\mapsto w+1\>\}\subset R$ .
Then we can write, where $\chi_+$ is the characteristic function of $\{ \Re z>0\}$ and $\chi_-$ the one of $\{ \Re z \leq 0\}$,
\begin{eqnarray*}
A_{l,q}(s,z)&=&\sum_{\<\sigma_l\>\backslash \Gamma_l} \left(\frac{\gamma'(z)}{\gamma(z)}\right)^q \sin^{s-q}\theta(\gamma z)\, dz^q\\
 &=&y(z)^{s-q}(\chi_+ +\sum_D 
\frac{\gamma'(z)^q}{\gamma(z)^q}\frac{|\gamma'(z)|^{s-q}}{|\gamma(z)|^{s-q}})dz^q+y(z)^{s-q}(\chi_- +\sum_G \frac{\gamma'(z)^q}{\gamma(z)^q}\frac{|\gamma'(z)|^{s-q}}{|\gamma(z)|^{s-q}})dz^q
\end{eqnarray*}
 and the $q$-form on ${\mathcal S}_l^{(r)}$:
\begin{eqnarray*}
 A_{l,q}^R(s,z)&=&y(z)^{s-q}(\chi_+ +\sum_D 
\frac{\gamma'(z)^q}{\gamma(z)^q}\frac{|\gamma'(z)|^{s-q}}{|\gamma(z)|^{s-q}})dz^q\\
&=&l^q(\sin l\Im w)^{s-q}(\chi +\sum_R (\tilde{\gamma}'w)^q|\tilde{\gamma}'w|^{s-q})dw^q, 
\end{eqnarray*}
where $\chi$ is the characteristic function of $w(\Re z >0\})$.

\begin{thm}\label{auto}
Let $S_l=\Gamma_l\backslash H$ be a family of geometrically finite hyperbolic surfaces 
degenerating to the surface $S$ with only one geodesic $c_l$
being pinched and  with $S_l$ having only one funnel $F_1$; $\Gamma_l$ contains the transformation
$\sigma_l(z)=e^lz$ corresponding to $c_l$ and the right half-collar for $c_l$ is in  $S_l\backslash F_1$. Let $S_0=\Gamma\backslash H$ be the component of $S$ containing
$p$,  the cusp arising from the right half-collar for $c_l$, $p=\infty$.
 Let $w=\frac{1}{l}\log z$ and write again with a little misuse of notation $A_{l,q}(s,w)= a_{l,q}(s,w)\, dw^q$ and the corresponding $q$-automorphic form  (for $\tilde{\Gamma}_l$)
${\hat a}_{l,q}(s,w)=\Im(w)^qa_{l,q}(s,w)$. Denote ${\hat a}_{l,q}^R$ the right half of  ${\hat a}_{l,q}$. Then the family
$(\frac{1}{l^s}{\hat a}_{l,q}^R(s,\pi_l(.)))_l$ converges uniformly on compact subsets of $S_0$ and on compact subsets of $\Re s>1$ to the Eisenstein series for weight $2q$:
$$E_{\infty,q}(s,w)=\sum_{\Gamma_\infty\backslash \Gamma} (\Im \gamma w)^s \left(\frac{c\bar{w}+d}{cw+d}\right)^q\, .$$
\end{thm}
\begin{rem}
We have the analogue result: if we denote ${\hat a}_{l,q}^L$
the left  half of  ${\hat a}_{l,q}$, then
$\frac{1}{l^s}{\hat a}_{l,q}^L(s,\pi_l(.))$ converges to $(-1)^q E_{\A,q}(s,.)$, where $E_{\A,q}(s,.)$ is the Eisenstein series of weight $2q$ (see definition \ref{weight}) associated to $\A$, the other cusp arising from the left half-collar for $c_l$.
\end{rem}
Before proving this theorem we give some complements and its corollaries.

\medskip

For $\Re s>1$ let $b_q(s)=e^{i\pi q/2}\int_0^\pi(\sin u)^{s-2} e^{-iqu}\, du$. Note that $b_1(s)=k(s-1)$. The function $b_q$ has the following properties (see e.g. \cite{kramer}): $ b_q(s+2)=\displaystyle\frac{s(s-1)}{s^2-q^2}b_q(s)$,  $b_q$ admits a meromorphic continuation to all $s\in\CC$, more precisely
$$b_q(s)=\pi 2^{-s+2}\frac{\Gamma(s-1)}{\Gamma(\frac{s+q}{2})\Gamma(\frac{s-q}{2}
)}\, .$$
In order to be consistent with the definition of Kudla and Millson's hyperbolic Eisenstein series, we may use the normalized $q-$automorphic forms 
\begin{equation}\label{nor}
\Xi_{l,q}(s,z)=\frac{1}{b_q(s)}A_{l,q}(s,z).
\end{equation}
We have indeed $\Xi_{l,1}(s,z)=\Theta(s-1, z)$ defined in (\ref{lien}) section 2.1.\\
We recall that the series (\ref{nor}) converges absolutely and locally uniformly for any $z\in H$ and $s\in \CC$ with $\Re s>1$, and that it is invariant with respect to $\Gamma$. A straightforward computation shows that the series $A_{l,q}(s,z)$ satisfies the  functional  differential  equation:
\begin{equation*}
\Delta_q^{\pm}A_{l,q}(s,z)-s(1-s)A_{l,q}(s,z)=(s+q)(s-q)A_{l,q}(s+2,z),
\end{equation*}
and the series (\ref{nor})
\begin{equation*}
\Delta_q^{\pm}\Xi_{l,q}(s,z)-s(1-s)\Xi_{l,q}(s,z)=s(s-1)\Xi_{l,q}(s+2,z).
\end{equation*}

\begin{prop}
The  series $A_{l,q}(s,z)$(resp. $\Xi_{l,q}(s,z) $) admits a meromorphic continuation to all of $\CC$.
\end{prop}
\begin{proof}
There are different ways to prove this; one is to use the  functional differential  equation (\ref{nor}) and to apply the method developed in \cite{kmi} (see also \cite{kramer}). More precisely let $\tilde{a}_{l,q}(s,z)$ (resp. $\tilde{ f}_{l,q}(s,z)$) the $q-$automorphic form associated to $A_{l,q}(s,z)$ (resp. $\Xi_{l,q}(s,z)$). We have 
\begin{equation*}
\Delta_{2q}\tilde{a}_{l,q}(s,z)+s(1-s)\tilde{a}_{l,q}(s,z)=(s+q)(q-s)\tilde{a}_{l,q}(s+2,z),
\end{equation*}
and 
\begin{equation}
\Delta_{2q}\tilde{f}_{l,q}(s,z)+s(1-s)\tilde{ f}_{l,q}(s,z)=s(1-s)\tilde{ f}_{l,q}(s+2,z);
\end{equation}
in other words
$$\tilde{ f}_{l,q}(s,z)=s(1-s)\int_{\mathcal D}G_s(z,z',q)\tilde{ f}_{l,q}(s+2,z')\, d\mu(z')\, .$$

\smallskip
Another way is to use \cite{fayreine}. We precise another calculation we will develop further. We can rewrite $\tilde{a}_{l,q}(s,z)$ as
$$\tilde{a}_{l,q}(s,z)=\sum_{\Gamma_l\backslash \Gamma}\left( \frac{c\bar{z}+d}{cz+d}\right)^q\left(\frac{\overline{\gamma z}}{\gamma z}\right)^{q/2}\left(\frac{\Im \gamma z}{|\gamma z|}\right)^s\, $$
and using the Fourier development of $G_s(z,z',q)$ and the expansion of $\int_1^{e^l}G_s(z,iy',q)\, d\ln y'\, $ we obtain the result (see \cite{fayreine} corollary 4.2 p.188).
\end{proof}
We can then refine the theorem \ref{auto} as
\begin{thm}\label{vitali}
 Let $\lambda_k^q$, $1\leq k\leq n$, be the eigenvalues of the laplacian $\Delta_{2q}$ on $S_0$ such that $0<\lambda_1 ^q< \lambda_2^q<... <\lambda_n^q<1/4\leq \lambda_{n+1}^q$ and the corresponding $s_k^q=\frac{1}{2} +\sqrt{\frac{1}{4}-\lambda_k^q}$ with $\Re s_k^q \geq 1/2$ and let $\Omega=\{\Re s>1/2\}\backslash \{s_1^q,...,s_n^q\}$.
\begin{enumerate}
\item In the case the geodesic $c_l$ is not separating, without loss of generality we can suppose that the two limiting cusps are represented by $\infty$ and $0$,  the family $(\frac{1}{l^s}{\hat f}_{l,q}(s,\pi_l(.)))_l$ converges uniformly on compact subsets of $S_0$ and on compact subsets of $\Omega$ to 
$\displaystyle\frac{1}{b_q(s)}E_{\infty, q}(s,.)+\frac{ (-1)^q}{b_q(s)} E_{0, q}(s,.)$. \label{one}
\item In the case $c_l$ is the geodesic boundary of a funnel, the family $(\frac{1}{l^s}{\hat f}_{l,q}(s,\pi_l(.)))_l$ converges uniformly on compact subsets of $S_0$ and $\Omega$ to  $\displaystyle\frac{1}{b_q(s)}E_{\infty, q}(s,.)$.\label{two}
\end{enumerate}
\end{thm}
\begin{rem}
We  remark that as $\sigma_l$ is the geodesic in the funnel,
then $D=\<\sigma_l\>\backslash (\Gamma_l - \<\sigma_l\>)$ and  $G=\emptyset$.

\end{rem}
We have seen that in the infinite volume case $s=1$ is a regular value for $E_{\infty, q}(s,z)$ for every $q\in \NN$. In the finite volume case  we know that $s=1$ is a simple pole for $E_{\infty, 0}(s,z)$, which is no more the case for $q\geq 1$:
\begin{rem}
For $q\in \NN^*$,  $s=1$ is a regular value of  $E_{\infty, q}(s,z)$.
\end{rem}
\begin{proof}
It comes from the Fourier development of $E_{\infty, q}(s,z)$, we have (\cite{fayreine}, p.175)
$$E_{\infty, q}(s,z)=y^s+\phi_q(s) y^{1-s}+\sum_{m\not=0}a_m(y,s)_q e^{2i\pi m x}\, ,$$
with $\phi_q(s)=\displaystyle\frac{\Gamma^2(s)}{\Gamma(s-q)\Gamma(s+q)}e^{-\pi i q}\phi_0(s)$. The function $\phi_0$ has a simple pole at $s=1$ (with residue $\text{Vol}(M_0)/\pi$) and $1/\Gamma(s-q)$ has a zero at $s=1$.
\end{proof}
We know give the poles of $b_q(s)$:
\begin{lemma}
For an even $q$, the poles of $b_q(s)$ are simples and at the numbers $s=1-2k$, $k\in \NN$.\\
For an odd $q$,  the poles of $b_q(s)$ are simples and at the numbers $s=-2k$, $k\in \NN$.
\end{lemma}
From the preceding results and theorem (\ref{vitali}) we deduce
\begin{cor}
With the preceding notations,
\begin{enumerate}
\item Suppose $q$ is odd. 	If the geodesic $c_l$ is not separating (resp.  the geodesic boundary of a funnel)  the family $(\frac{1}{l^s}{\hat a}_{l,q}(s,\pi_l(.)))_l$ converges uniformly on compact subsets of $S_0$ and on compact subsets of $\Omega$ to 
$E_{\infty, q}(s,.)+(-1)^q E_{0, q}(s,.)$ (resp. $E_{\infty, q}(s,.)$). 
\item Suppose $q$ is even, we have the same preceding results replacing $\Omega$ by $\Omega \backslash \{1\}$.
\end{enumerate}
\end{cor} 
\begin{rem}
For $q=1$ we obtain the result announced in the introduction.
\end{rem}
\subsection{ Proofs of the main theorems and final remarks.\\ \smallskip}


{\it Proof of theorem \ref{auto}.}

It is enough to demonstrate the convergence for $\beta$ a
relatively compact set in the fundamental domain $\mathcal F$.
Given $\epsilon>0$ denote $G_\epsilon$ the set of cosets and
representatives for $\<z\mapsto z+1\>\backslash \Gamma$ such that $\sup
\Im A(\mathcal F)<\epsilon$  for $[A]\not \in G_\epsilon$ and let $R_l$ be
the corresponding cosets of $\<z\mapsto z+1\>\backslash \tilde{\Gamma}_l$
with corresponding representatives. The set $G_\epsilon$ is finite.\\ The cosets $G_\epsilon$ of
$\Gamma$ satisfy (modulo the $\<z\mapsto z+1\>$ action) $\{0\leq \Re
w<1, \Im w>\epsilon\}\subset \cup_{A\in G_\epsilon}A(\mathcal F)$; thus
for $l$ sufficiently small the cosets  $R_l$ satisfy (modulo the
$\<z\mapsto z+1\>$ action) $\{0\leq \Re w<1, 2\epsilon<\Im
w<\pi/2l\}\subset \cup_{A\in R_l}A({\mathcal F}_l)$   (a
consequence of the convergence on compact subsets of the
$\tilde{f}_l$ and that ${\mathcal F}_l$ contains the
right half-collar for $c_l$ , $\{0\leq \Re w<1, c(l)\leq\Im
w<\pi/2l\}$). Now for a right coset $[A]\in R- R_l$ then
$A({\mathcal F}_l)$ lies below the $c_l$ geodesic $\{\Im
w=\pi/2l\}$ and is disjoint modulo the $\<z\mapsto z+1\>$ action from
$\{0\leq \Re w<1, 2\epsilon<\Im w<\pi/2l\}$, since the latter is
covered by the $R_l$ cosets. Thus for $[A]\in R-R_l$, modulo the
$\<z\mapsto z+1\>$ action, then $A({\mathcal F}_l)\subset \{ 0\leq \Re
w<1, \Im w<2\epsilon\}$.
For $w\in\beta$ we write 
\begin{eqnarray*}
\frac{1}{l^s}{\hat a}_{l,q}^R(s,w)&=& \frac{1}{l^{s-q}}(\Im w)^q(\sin l\Im w)^{s-q}(\chi +\sum_R (\tilde{\gamma}'w)^q|\tilde{\gamma}'w|^{s-q}) \\
&=&\frac{1}{l^{s-q}}(\Im w)^q(\sin l\Im w)^{s-q}(\chi +\sum_{R_l} (\tilde{\gamma}'w)^q|\tilde{\gamma}'w|^{s-q})\\
&+&\frac{1}{l^{s-q}}(\Im w)^q(\sin l\Im w)^{s-q}(\sum_{R-R_l} (\tilde{\gamma}'w)^q|\tilde{\gamma}'w|^{s-q}).
\end{eqnarray*}

 Because of lemma \ref{lem}, the principal problem lies in estimating the second sum: it remains to show that
 $$\lim_{l\to 0}\sum_{R-R_l} |\tilde{\gamma}'w|^\sigma=0\, ,$$
 where $\sigma=\Re s$.
$$\sum_{R-R_l} |\tilde{\gamma}'w|^\sigma=\sum_{w^{-1}(R-R_l)w}|z_l\frac{\gamma' z_l}{\gamma z_l}|^{\sigma}\leq |z_l|^\sigma \sum_{w^{-1}(R-R_l)w}|\gamma' z_l|^{\sigma}
\, ,$$
where $w=\frac{1}{l}\log z_l$ and $\gamma {\mathcal F}_l\subset w^{-1}(\{ 0\leq \Re
w<1, \Im w<2\epsilon\})$.
We deduce
$$\sum_{R-R_l} |\tilde{\gamma}'w|^\sigma\leq \frac{1}{\sin^{\sigma}(l\Im w)}\sum_{w^{-1}(R-R_l)w}\Im^{\sigma}(\gamma z_l)\, .$$
Let $\epsilon_0\in]0,\sinh^{-1}1[$ such that for $l\leq l_0(\epsilon)$ and $ w\in\beta$, $B(z_l,\epsilon_0)\subset {\mathcal F}_l$.
  Then there exists $\Lambda_{\epsilon_0}$ independent of $z_0$ so that
$$\int_{B(z_0,\epsilon_0)}y^\sigma\frac{dxdy}{y^2}=\Lambda_{\epsilon_0}y(z_0)^\sigma\, .$$
So $$\sum_{w^{-1}(R-R_l)w}\Im^{\sigma}(\gamma z_l)^{\sigma}=\frac{1}{\Lambda_{\epsilon_0}}\sum_{w^{-1}(R-R_l)w}\int_{B(\gamma z_l,\epsilon_0)}y^\sigma\frac{dxdy}{y^2}\, .$$
Now (see e.g. \cite{w2} p.102):
\begin{lemma}
The multiplicity of the projection map $H\rightarrow \Gamma\backslash
H$ restricted   to $B(z_0,\eta)$ with $2\eta<c_0$ is at
most $M\rho^{-2}(z_0)$, where $M$  is a constant and $\rho(z_0)$ the
injectivity radius at $z_0$.
\end{lemma}
We give a detailed proof for the convenience of the reader.
\begin{proof} If $B(z_0,\eta)\cap B(\gamma
z_0,\eta)\not =\emptyset,\, \gamma\in \Gamma$, then
$d(z_0,\gamma z_0)<2\eta<c_0$ and $z_0$ is in a cusp region or
the collar for a short geodesic. Let $c=c_0/2$, then we have
$\rho(z_0)<c$. Set $m(\eta)$ the multiplicity of the projection
restricted   to $B(z_0,\eta)$. As $2\eta +\rho(z_0)<3c$ and the
$B(\gamma z_0,\rho(z_0))$ are disjoints we have
$$m(\eta)\mu(B(z_0,\rho(z_0))\leq \mu(B(z_0,3c))\, .$$
So $m(\eta)\leq \frac{\cosh 3c -1}{\cosh \rho(z_0) -1}\leq
2(\cosh 3c -1)\rho(z_0)^{-2}$.
\end{proof}
Then we have
$$ \sum_{w^{-1}(R-R_l)w}\Im^{\sigma}(\gamma z_l)^{\sigma}\leq \frac{A(c_0)}{\Lambda_{\epsilon_0}}\rho^{-2}(z_l)\frac{e^{l\sigma}-1}{\sigma}\frac{(2\epsilon l)^{\sigma-1}}{\sigma-1}\, .$$
Moreover as $B(z_l,\epsilon_0)\subset {\mathcal F}_l$, $\rho(z_l)\geq \epsilon_0$ and the conclusion follows.

\bigskip

\noindent{\it Proof of theorem \ref{vitali}.}

The main point of the proof is to use Vitali's theorem as in the finite volume case.
More precisely, first  recall the 
\begin{defn}
A family $(f_n)$ of meromorphic functions on a domain $\mathcal O$ in $\mathbb C$ is said bounded in $\mathcal O$ if
\begin{enumerate}
\item $\exists (z_m)_m$ a discrete subset in $\mathcal O$,
\item $\forall K$ compact set of ${\mathcal O}\backslash \{z_m\}$, $\exists n(K)$, $\forall n\geq n(K)$, $f_n$ has no pole in $K$,
\item $M_K=\sup_{n\geq n(K)}(\sup_{z\in K}|f_n(z)|)< +\infty$.
\end{enumerate}
\end{defn}
As before $\beta$ will denote a relatively compact set in the fundamental domain $\mathcal F$ of $S_0$. The family $(\frac{1}{l^s}{\hat f}_{l,q}(s,\pi(w)))_l$, where ${\hat f}_{l,q}(s,w)=\frac{1}{b_q(s)}{\hat a}_{l,q}(s,w)$ is a family of meromorphic functions on $\Omega=\{\Re s>1/2\}\backslash \{s_1^q,...,s_n^q\} $; because of theorem (\ref{auto}), for every compact $C\subset \{\Re s>1\}$   it converges uniformly (a fortiori simply) on $C\times \beta \subset \{\Re s>1\}\times {\mathcal F}$ to
\begin{enumerate}
\item  $\frac{1}{b_q(s)}(E_{\infty, q}(s,.)+(-1)^q E_{0, q}(s,.))$ in  case (\ref{one}) of theorem (\ref{auto}),
\item $\frac{1}{b_q(s)} E_{\infty, q}(s,.)$ in case (\ref{two}).
\end{enumerate}
We'll show that for every $K$  compact set in  $\Omega$, there exists $l(K)$ such that for all $l$ less than $l(K)$, $s\mapsto  \frac{1}{l^s}{\hat f}_{l,q}(s,\pi(.))$ has no pole on $K$ and 
$M_K=\sup_{l\leq l(K)}(\sup_{z\in K}||\frac{1}{l^s}{\hat f}_{l,q}(s,\pi(.)) ||_{\infty, \beta})< +\infty$ then we can conclude.\\

The frame of the proof being independent of $q$, we will suppose $q=0$ and only indicate the notable changes in the general $q$ case, when there are. 
As in the finite volume case we start from the equality
$$ \tilde{ f}_{l}(s,z)/l^s=s(1-s)\int_{\mathcal D}G_s(z,z')\tilde{ f}_{l}(s+2,z')/l^s\, d\mu(z')\, ;  $$
where, to simplify the notations, we omit to write $q=0$.\\
We will denote $C$ a compact subset of $S_0$,  $Y\subset S_0$ such that $\pi_l(Y)$ is a standard collar and $\pi_l(Y)\cap\pi_l(C)=\emptyset$, $X=S_0\backslash Y$, $K'$ a compact subset of $\Omega'=\Omega\backslash \{1\}$. 
Cut the preceding integral as $I_l+J_l$ where
$$I_l= \int_{\pi_l(Y)}G_s(z,z')\tilde{ f}_{l}(s+2,z')/l^s\, d\mu(z')\, ,\,\,  \,\, \,  \, \, \, J_l=\int_{\pi_l(X)}G_s(z,z')\tilde{ f}_{l}(s+2,z')/l^s\, d\mu(z')\,\,.$$
The main step is now to estimate $ G_s(z,z',q)  $. We apply \cite{fayreine} Corollary 1.3 p.151:
\begin{prop}\label{mean}
If $\Delta_{2q}f=s(s-1)f$ in some non-Euclidien disc $B_r$ of radius $r$ about $z_0\in H$, then $f$ has the mean value property:
$$f(z_0)=\frac{1}{m_q(r,s)}\int_{B_r}f(z)\left(\frac{z-\bar{z_0}}{z_0-\bar{z}}\right)^q d\mu(z)\,
\text{ where }\, 
m_q(r,s)=2\pi \int_1^r\sh r P_{s,q}(r)\approx\pi r^2\,\,\,\, \text{as}\,\,\, r\to 0\, ;$$
\end{prop}
to conclude that $\forall w\in C$, $\forall z'\in\pi_l(Y)$, $\forall s\in K'$, $|G_s(\pi_l(w),z',q)|\leq O(1) ||G_s(\pi_l(w),.,q)||_{L_2}$  (see also \cite{preprint} Lemme 4.4).\\
 The next point is to show that $\forall (w,s)\in C\times K'$, $||G_s(\pi_l(w),.,q)||_{L_2}=O(1)$ (see also \cite{preprint} Lemme 4.5), the constants in question depending only of the compact sets. For this we use the following spectral decomposition:
\begin{lemma}
For $a>1$ 
$$G_s(\pi_l(w),z,l)=G_a(\pi_l(w),z,l)+ \sum_n \left[\frac{1}{s(1-s)-\lambda_n}-\frac{1}{a(1-a)-\lambda_n}\right]\phi_n(\pi_l(w)\overline{\phi_n(z)}+$$
$$\frac{1}{4\pi i}\int_{-\infty}^{+\infty}\left[\frac{1}{s(1-s) -(1/4 +t^2)}-\frac{1}{a(1-a) -(1/4+t^2)}\right]E_c(\pi_l(w),\frac{1}{2}+it)\overline{E_c(z,\frac{1}{2}+it)}\, dt +$$
$$\frac{1}{4\pi i}\int_{-\infty}^{+\infty}\left[\frac{1}{s(1-s) -(1/4 +t^2)}-\frac{1}{a(1-a) -(1/4 +t^2)}\right]\int_1^{\lambda_f^2}E_f(\pi_l(w),\frac{1}{2}+it,b)\overline{E_f(z,\frac{1}{2}+it,b)}\, db\, dt\, .$$
\end{lemma}
From which we deduce, with the notation $\lambda_u=u(1-u)$

\begin{lemma}
$$||G_s(\pi_l(w),.,l)-G_a(\pi_l(w),.,l)||_{L^2}^2=|\lambda_a -\lambda_s|^2\sum_n\frac{|\phi_n(\pi_l(w)|^2}{|\lambda_s-\lambda_n(l)|^2|\lambda_a-\lambda_n(l)|^2} +$$
$$|\lambda_a -\lambda_s|^2\frac{1}{4\pi}\int_{-\infty}^{+\infty}\frac{|E_c(\pi_l(w),1/2 +it)|^2}{|\lambda_s-(1/2+t^2)|^2|\lambda_a-(1/2+t^2)|^2}\, dt +$$
$$ |\lambda_a -\lambda_s|^2\frac{1}{4\pi}\int_{-\infty}^{+\infty}\frac{1}{|\lambda_s-(1/2+t^2)|^2|\lambda_a-(1/2+t^2)|^2}\int_1^{\lambda_f^2}|E_f(\pi_l(w),1/2 +it,b)|^2\,db\, dt  \, .$$
\end{lemma}
So there exists a constant $M$ depending only on the compact $K'$ of $\Omega'$ such that
$$||G_s(\pi_l(w),.,l)-G_a(\pi_l(w),.,l)||_{L^2}^2\leq \frac{|\lambda_a -\lambda_s|^2}{M}||G_a(\pi_l(w),.,l)||_{L^2}^2\, .$$
This allows for concluding that $I_l=O(1)$ and $J_l=O(l^2)$.

\begin{rem} 1) We refer also to the result of J.Fay
\cite{jfay} final remark p.201-202. \\
2) It is possible to establish a link  between hyperbolic Eisenstein series and generalized eigenfunctions- I thank F.Naud for this suggestion- and to use the first remark to conclude that as $l\to 0$
$$\int_0^l E^f_l(s,z,t)e^{ts}\, dt \to E_\infty(s,z)\, .$$ \end{rem}

{\bf Application.}
One of the application we can think about is in studying the degeneration of the residues of the hyperbolic Eisenstein series. Let us give an example: we look at the case of a degenerating family of compact Riemann surfaces, a non-separating geodesic being pinched, with the family of scalar-valued hyperbolic Eisenstein series degenerating. Remember that we obtained that the series $\frac{1}{l^s} A_l(s,z)$ converges uniformly on compact of $S_0$ to $E_\infty(s,z) + E_0(s,z)$. The last sum of Eisenstein series has no poles on $[1/2,1]$ except at $s=1$ and  for a finite number of $s_k=\frac{1}{2}+\sqrt{\frac{1}{4}-\lambda_k}$ where the $\lambda_k$ correspond to the residual spectrum. In other words if $(\lambda_k(l))$ converges to $\lambda_k$
 where $\lambda_k$ is a small cuspidal eigenvalue, then $Res(\frac{1}{l^s} A_l(s,z))\to 0$, otherwise $Res(\frac{1}{l^s} A_l(s,z))\to Res(E_\infty(s,z) + E_0(s,z))$. There are many restrictions to go ahead with calculation. First of all we are only dealing with small eigenvalues and so here  there is no hope to characterize embedded eigenvalues through degeneration. Moreover we need to take into account the multiplicity of an eigenvalue $\lambda_k(l)$. So the easiest result we can obtained is a characterization of the residual spectrum -remember that an eigenvalue in the residual spectrum is simple- with $\alpha_k(l)=\int_{c_l}\psi_{k,l}(z)\, d\mu(z)$ where the eigenfunction $\psi_{k,l}$ is associated to the eigenvalue $\lambda_k(l)$, if $\lambda_k$ is a pole of $E_\infty(s,z) + E_0(s,z)$ then
 $\alpha_k(l)=O(l^{1/2 +\sqrt{1/4 -\lambda_k(l)}})$
 and otherwise $\alpha_k(l)=o(l^{1/2 +\sqrt{1/4 -\lambda_k(l)}})$.\\
 

\end{document}